\documentclass[11pt,reqno]{amsart}
\usepackage{eurosym}
\usepackage{amsfonts}
\usepackage{amsmath}
\usepackage[bookmarksnumbered, colorlinks, plainpages]{hyperref}
\usepackage{graphicx}
\usepackage{epsfig}
\usepackage{epstopdf}
\usepackage{color}
\usepackage{caption}
\usepackage{subcaption}

\hypersetup{colorlinks=true,linkcolor=red, anchorcolor=green, citecolor=cyan, urlcolor=red, filecolor=magenta, pdftoolbar=true}
\newtheorem{theorem}{Theorem}[section]
\newtheorem{lemma}[theorem]{Lemma}

\newtheorem{corollary}[theorem]{Corollary}
\theoremstyle{definition}
\newtheorem{definition}[theorem]{Definition}
\newtheorem{example}[theorem]{Example}

\theoremstyle{remark}

\numberwithin{equation}{section}

\begin{document}
\date{{\scriptsize Received: , Accepted: .}}
\title[A New Solution to the Rhoades' Open Problem]{A New Solution to the
Rhoades' Open Problem with an Application}

\begin{abstract}
We give a new solution to the Rhoades' open problem on the discontinuity at
fixed point via the notion of an $S$-metric. To do this, we inspire with the
notion of a Zamfirescu mapping. Also, we consider a recent problem called
the \textquotedblleft fixed-circle problem\textquotedblright\ and propose a
new solution to this problem as an application of our technique.
\end{abstract}

\author[N. \"{O}ZG\"{U}R]{N\.{I}HAL \"{O}ZG\"{U}R}
\address[Nihal \"{O}zg\"{u}r]{Bal\i kesir University, Department of
Mathematics, 10145 Bal\i kesir, TURKEY}
\email{nihal@balikesir.edu.tr}
\author[N. TA\c{S}]{N\.{I}HAL TA\c{S}}
\address[Nihal Ta\c{s}]{Bal\i kesir University, Department of Mathematics,
10145 Bal\i kesir, TURKEY}
\email{nihaltas@balikesir.edu.tr}
\maketitle


\setcounter{page}{1}


\section{\textbf{Introduction and Preliminaries}}

\label{sec:intro} Fixed-point theory has been extensively studied by various
aspects. One of these is the discontinuity problem at fixed points (see \cite%
{Bisht-2017-1, Bisht-2017-2, Bisht-3, Bisht and Rakocevic, Bisht-Tbilisi,
Bisht 2019, PantA, Pant, Pant-Ozgur-Tas, Belgium} for some examples).
Discontinuous functions have been widely appeared in many areas of science
such as neural networks (for example, see \cite{Forti-2003, Liu-2011,
Lu-2005, Lu-2006}). In this paper we give a new solution to the Rhoades'
open problem (see \cite{Rhoades} for more details) on the discontinuity at
fixed point in the setting of an $S$-metric space which is a recently
introduced generalization of a metric space. $S$-metric spaces were
introduced in \cite{Sedgi-Shobe-Aliouche} by Sedgi et al., as follows:

\begin{definition}
\cite{Sedgi-Shobe-Aliouche} \label{def5} Let $X$ be a nonempty set and $%
\mathcal{S}:X\times X\times X\rightarrow \lbrack 0,\infty )$ be a function
satisfying the following conditions for all $x,y,z,a\in X:$

$S1)$ $\mathcal{S}(x,y,z)=0$ if and only if $x=y=z$,

$S2)$ $\mathcal{S}(x,y,z)\leq \mathcal{S}(x,x,a)+\mathcal{S}(y,y,a)+\mathcal{%
S}(z,z,a)$.

Then $\mathcal{S}$ is called an $S$-metric on $X$ and the pair $(X,\mathcal{S%
})$ is called an $S$-metric space.
\end{definition}

Relationships between a metric and an $S$-metric were given as follows:

\begin{lemma}
\label{lem4} \cite{Hieu} Let $(X,d)$ be a metric space. Then the following
properties are satisfied$:$

\begin{enumerate}
\item $\mathcal{S}_{d}(x,y,z)=d(x,z)+d(y,z)$ for all $x,y,z\in X$ is an $S$%
-metric on $X$.

\item $x_{n}\rightarrow x$ in $(X,d)$ if and only if $x_{n}\rightarrow x$ in
$(X,\mathcal{S}_{d})$.

\item $\{x_{n}\}$ is Cauchy in $(X,d)$ if and only if $\{x_{n}\}$ is Cauchy
in $(X,\mathcal{S}_{d}).$

\item $(X,d)$ is complete if and only if $(X,\mathcal{S}_{d})$ is complete.
\end{enumerate}
\end{lemma}

The metric $\mathcal{S}_{d}$ was called as the $S$-metric generated by $d$
\cite{Ozgur-mathsci}. Some examples of an $S$-metric which is not generated
by any metric are known (see \cite{Hieu, Ozgur-mathsci} for more details).

Furthermore, Gupta claimed that every $S$-metric on $X$ defines a metric $%
d_{S}$ on $X$ as follows$:$%
\begin{equation}
d_{S}(x,y)=\mathcal{S}(x,x,y)+\mathcal{S}(y,y,x),  \label{ds}
\end{equation}%
for all $x,y\in X$ \cite{Gupta}. However, since the triangle inequality does
not satisfied for all elements of $X$ everywhen, the function $d_{S}(x,y)$
defined in (\ref{ds}) does not always define a metric (see \cite%
{Ozgur-mathsci}).

In the following, we see an example of an $S$-metric which is not generated
by any metric.

\begin{example}
\cite{Ozgur-mathsci} \label{exm:S-metric} Let $X=%
\mathbb{R}
$ and the function $\mathcal{S}:X\times X\times X\rightarrow \lbrack
0,\infty )$ be defined as%
\begin{equation*}
\mathcal{S}(x,y,z)=\left\vert x-z\right\vert +\left\vert x+z-2y\right\vert
\text{,}
\end{equation*}%
for all $x,y,z\in
\mathbb{R}
$. Then $\mathcal{S}$ is an $S$-metric which is not generated by any metric
and the pair $(X,\mathcal{S})$ is an $S$-metric space.
\end{example}

The following lemma will be used in the next sections.

\begin{lemma}
\label{lem5} \cite{Sedgi-Shobe-Aliouche} Let $(X,\mathcal{S})$ be an $S$%
-metric space. Then we have%
\begin{equation*}
\mathcal{S}(x,x,y)=\mathcal{S}(y,y,x)\text{.}
\end{equation*}
\end{lemma}

In this paper, our aim is to obtain a new solution to the Rhoades' open
problem on the existence of a contractive condition which is strong enough
to generate a fixed point but which does not force the map to be continuous
at the fixed point. To do this, we inspire of a result of Zamfirescu given
in \cite{Zamfirescu}.

On the other hand, a recent aspect to the fixed point theory is to consider
geometric properties of the set $Fix(T)$, the fixed point set of the
self-mapping $T$. Fixed circle problem (resp. fixed disc problem) have been
studied in this context (see \cite{Bisht 2019, Ozgur-chapter, Ozgur-Tas-malaysian, Ozgur-Tas-circle-thesis, Ozgur-Tas-Celik, ozgur-aip, Ozgur-simulation, Pant-Ozgur-Tas, Belgium, Tas math, Tas}). As an application, we
present a new solution to these problems. We give necessary examples to
support our theoretical results.

\section{\textbf{Main Results}}

\label{sec:1} From now on, we assume that $(X,\mathcal{S})$ is an $S$-metric
space and $T:X\rightarrow X$ is a self-mapping. In this section we use the
numbers defined as%
\begin{equation*}
M_{z}\left( x,y\right) =\max \left\{ ad\left( x,y\right) ,\frac{b}{2}\left[
d\left( x,Tx\right) +d\left( y,Ty\right) \right] ,\frac{c}{2}\left[ d\left(
x,Ty\right) +d\left( y,Tx\right) \right] \right\}
\end{equation*}%
and%
\begin{equation*}
M_{z}^{S}\left( x,y\right) =\max \left\{
\begin{array}{c}
a\mathcal{S}\left( x,x,y\right) ,\frac{b}{2}\left[ \mathcal{S}\left(
x,x,Tx\right) +\mathcal{S}\left( y,y,Ty\right) \right] , \\
\frac{c}{2}\left[ \mathcal{S}\left( x,x,Ty\right) +\mathcal{S}\left(
y,y,Tx\right) \right]
\end{array}%
\right\} ,
\end{equation*}%
where $a,b\in \left[ 0,1\right) $ and $c\in \left[ 0,\frac{1}{2}\right] $.

We give the following theorem as a new solution to the Rhoades' open problem.

\begin{theorem}
\label{thm1} Let $(X,\mathcal{S})$ be a complete $S$-metric space and $T$ a
self-mapping on $X$ satisfying the conditions

$i)$ There exists a function $\phi :\mathbb{R}^{+}\rightarrow \mathbb{R}^{+}$
such that $\phi (t)<t$ for each $t>0$ and
\begin{equation*}
\mathcal{S}\left( Tx,Tx,Ty\right) \leq \phi \left( M_{z}^{S}\left(
x,y\right) \right) \text{,}
\end{equation*}%
for all $x,y\in X$,

$ii)$ There exists a $\delta =\delta \left( \varepsilon \right) >0$ such
that $\varepsilon <M_{z}^{S}\left( x,y\right) <\varepsilon +\delta $ implies
$\mathcal{S}\left( Tx,Tx,Ty\right) \leq \varepsilon $ for a given $%
\varepsilon >0$.

Then $T$ has a unique fixed point $u\in X$. Also, $T$ is discontinuous at $u$
if and only if $\underset{x\rightarrow u}{\lim }M_{z}^{S}\left( x,u\right)
\neq 0$.
\end{theorem}

\begin{proof}
At first, we define the number%
\begin{equation*}
\xi =\max \left\{ a,\frac{2}{2-b},\frac{c}{2-2c}\right\} .
\end{equation*}%
Clearly, we have $\xi <1$.

By the condition $(i)$, there exists a function $\phi :\mathbb{R}%
^{+}\rightarrow \mathbb{R}^{+}$ such that $\phi (t)<t$ for each $t>0$ and
\begin{equation*}
\mathcal{S}\left( Tx,Tx,Ty\right) \leq \phi \left( M_{z}^{S}\left(
x,y\right) \right) \text{,}
\end{equation*}%
for all $x,y\in X$. Using the properties of $\phi $, we obtain%
\begin{equation}
\mathcal{S}\left( Tx,Tx,Ty\right) <M_{z}^{S}\left( x,y\right) \text{,}
\label{eqno1}
\end{equation}%
whenever $M_{z}^{S}\left( x,y\right) >0$.

Let us consider any $x_{0}\in X$ with $x_{0}\neq Tx_{0}$ and define a
sequence $\left\{ x_{n}\right\} $ as $x_{n+1}=Tx_{n}=T^{n}x_{0}$ for all $%
n=0,1,2,3,...$. Using the condition $(i)$ and the inequality (\ref{eqno1}),
we get
\begin{eqnarray}
\mathcal{S}\left( x_{n},x_{n},x_{n+1}\right) &=&\mathcal{S}\left(
Tx_{n-1},Tx_{n-1},Tx_{n}\right) \leq \phi \left( M_{z}^{S}\left(
x_{n-1},x_{n}\right) \right)  \label{eqno2} \\
&<&M_{z}^{S}\left( x_{n-1},x_{n}\right)  \notag \\
&=&\max \left\{
\begin{array}{c}
a\mathcal{S}\left( x_{n-1},x_{n-1},x_{n}\right) , \\
\frac{b}{2}\left[ \mathcal{S}\left( x_{n-1},x_{n-1},Tx_{n-1}\right) +%
\mathcal{S}\left( x_{n},x_{n},Tx_{n}\right) \right] , \\
\frac{c}{2}\left[ \mathcal{S}\left( x_{n-1},x_{n-1},Tx_{n}\right) +\mathcal{S%
}\left( x_{n},x_{n},Tx_{n-1}\right) \right]%
\end{array}%
\right\}  \notag \\
&=&\max \left\{
\begin{array}{c}
a\mathcal{S}\left( x_{n-1},x_{n-1},x_{n}\right) , \\
\frac{b}{2}\left[ \mathcal{S}\left( x_{n-1},x_{n-1},x_{n}\right) +\mathcal{S}%
\left( x_{n},x_{n},x_{n+1}\right) \right] , \\
\frac{c}{2}\left[ \mathcal{S}\left( x_{n-1},x_{n-1},x_{n+1}\right) +\mathcal{%
S}\left( x_{n},x_{n},x_{n}\right) \right]%
\end{array}%
\right\}  \notag \\
&=&\max \left\{
\begin{array}{c}
a\mathcal{S}\left( x_{n-1},x_{n-1},x_{n}\right) , \\
\frac{b}{2}\left[ \mathcal{S}\left( x_{n-1},x_{n-1},x_{n}\right) +\mathcal{S}%
\left( x_{n},x_{n},x_{n+1}\right) \right] , \\
\frac{c}{2}\mathcal{S}\left( x_{n-1},x_{n-1},x_{n+1}\right)%
\end{array}%
\right\} .  \notag
\end{eqnarray}%
Assume that $M_{z}^{S}\left( x_{n-1},x_{n}\right) =a\mathcal{S}\left(
x_{n-1},x_{n-1},x_{n}\right) $. Then using the inequality (\ref{eqno2}), we
have
\begin{equation*}
\mathcal{S}\left( x_{n},x_{n},x_{n+1}\right) <a\mathcal{S}\left(
x_{n-1},x_{n-1},x_{n}\right) \leq \xi \mathcal{S}\left(
x_{n-1},x_{n-1},x_{n}\right) <\mathcal{S}\left( x_{n-1},x_{n-1},x_{n}\right)
\end{equation*}%
and so%
\begin{equation}
\mathcal{S}\left( x_{n},x_{n},x_{n+1}\right) <\mathcal{S}\left(
x_{n-1},x_{n-1},x_{n}\right) .  \label{eqno3}
\end{equation}%
Let $M_{z}^{S}\left( x_{n-1},x_{n}\right) =\frac{b}{2}\left[ \mathcal{S}%
\left( x_{n-1},x_{n-1},x_{n}\right) +\mathcal{S}\left(
x_{n},x_{n},x_{n+1}\right) \right] .$ Again using the inequality (\ref{eqno2}%
), we get%
\begin{equation*}
\mathcal{S}\left( x_{n},x_{n},x_{n+1}\right) <\frac{b}{2}\left[ \mathcal{S}%
\left( x_{n-1},x_{n-1},x_{n}\right) +\mathcal{S}\left(
x_{n},x_{n},x_{n+1}\right) \right] \text{,}
\end{equation*}%
which implies%
\begin{equation*}
\left( 1-\frac{b}{2}\right) \mathcal{S}\left( x_{n},x_{n},x_{n+1}\right) <%
\frac{b}{2}\mathcal{S}\left( x_{n-1},x_{n-1},x_{n}\right)
\end{equation*}%
and hence%
\begin{equation*}
\mathcal{S}\left( x_{n},x_{n},x_{n+1}\right) <\frac{b}{2-b}\mathcal{S}\left(
x_{n-1},x_{n-1},x_{n}\right) \leq \xi \mathcal{S}\left(
x_{n-1},x_{n-1},x_{n}\right) .
\end{equation*}%
This yields%
\begin{equation}
\mathcal{S}\left( x_{n},x_{n},x_{n+1}\right) <\mathcal{S}\left(
x_{n-1},x_{n-1},x_{n}\right) .  \label{eqno4}
\end{equation}

Suppose that $M_{z}^{S}\left( x_{n-1},x_{n}\right) =\frac{c}{2}\mathcal{S}%
\left( x_{n-1},x_{n-1},x_{n+1}\right) .$ Then using the inequality (\ref%
{eqno2}), Lemma \ref{lem5} and the condition $(S2)$, we obtain
\begin{eqnarray*}
\mathcal{S}\left( x_{n},x_{n},x_{n+1}\right) &<&\frac{c}{2}\mathcal{S}\left(
x_{n-1},x_{n-1},x_{n+1}\right) =\frac{c}{2}\mathcal{S}\left(
x_{n+1},x_{n+1},x_{n-1}\right) \\
&\leq &\frac{c}{2}\left[ \mathcal{S}\left( x_{n-1},x_{n-1},x_{n}\right) +2%
\mathcal{S}\left( x_{n+1},x_{n+1},x_{n}\right) \right] \\
&=&\frac{c}{2}\mathcal{S}\left( x_{n-1},x_{n-1},x_{n}\right) +c\mathcal{S}%
\left( x_{n+1},x_{n+1},x_{n}\right) \\
&=&\frac{c}{2}\mathcal{S}\left( x_{n-1},x_{n-1},x_{n}\right) +c\mathcal{S}%
\left( x_{n},x_{n},x_{n+1}\right) ,
\end{eqnarray*}%
which implies
\begin{equation*}
\left( 1-c\right) \mathcal{S}\left( x_{n},x_{n},x_{n+1}\right) <\frac{c}{2}%
\mathcal{S}\left( x_{n-1},x_{n-1},x_{n}\right) .
\end{equation*}%
Considering this, we find%
\begin{equation*}
\mathcal{S}\left( x_{n},x_{n},x_{n+1}\right) <\frac{c}{2\left( 1-c\right) }%
\mathcal{S}\left( x_{n-1},x_{n-1},x_{n}\right) \leq \xi \mathcal{S}\left(
x_{n-1},x_{n-1},x_{n}\right)
\end{equation*}%
and so
\begin{equation}
\mathcal{S}\left( x_{n},x_{n},x_{n+1}\right) <\mathcal{S}\left(
x_{n-1},x_{n-1},x_{n}\right) .  \label{eqno5}
\end{equation}%
If we set $\alpha _{n}=\mathcal{S}\left( x_{n},x_{n},x_{n+1}\right) $, then
by the inequalities (\ref{eqno3}), (\ref{eqno4}) and (\ref{eqno5}), we find%
\begin{equation}
\alpha _{n}<\alpha _{n-1},  \label{eqno6}
\end{equation}%
that is, $\alpha _{n}$ is strictly decreasing sequence of positive real
numbers whence the sequence $\alpha _{n}$ tends to a limit $\alpha \geq 0$.

Assume that $\alpha >0$. There exists a positive integer $k\in \mathbb{N}$
such that $n\geq k$ implies%
\begin{equation}
\alpha <\alpha _{n}<\alpha +\delta (\alpha ).  \label{eqno7}
\end{equation}%
Using the condition $(ii)$ and the inequality (\ref{eqno6}), we get%
\begin{equation}
\mathcal{S}\left( Tx_{n-1},Tx_{n-1},Tx_{n}\right) =\mathcal{S}\left(
x_{n},x_{n},x_{n+1}\right) =\alpha _{n}<\alpha ,  \label{eqno8}
\end{equation}%
for $n\geq k$. Then the inequality (\ref{eqno8}) contradicts to the
inequality (\ref{eqno7}). Therefore, it should be $\alpha =0$.

Now we prove that $\left\{ x_{n}\right\} $ is a Cauchy sequence. Let us fix
an $\varepsilon >0$. Without loss of generality, we suppose that $\delta
\left( \varepsilon \right) <\varepsilon $. There exists $k\in \mathbb{N}$
such that%
\begin{equation*}
\mathcal{S}\left( x_{n},x_{n},x_{n+1}\right) =\alpha _{n}<\frac{\delta }{4}%
\text{,}
\end{equation*}%
for $n\geq k$ since $\alpha _{n}\rightarrow 0$. Using the mathematical
induction and the Jachymski's technique (see \cite{Jach-1, Jach-2} for more
details) we show
\begin{equation}
\mathcal{S}\left( x_{k},x_{k},x_{k+n}\right) <\varepsilon +\frac{\delta }{2}%
\text{,}  \label{eqno9}
\end{equation}%
for any $n\in \mathbb{N}$. At first, the inequality (\ref{eqno9}) holds for $%
n=1$ since
\begin{equation*}
\mathcal{S}\left( x_{k},x_{k},x_{k+1}\right) =\alpha _{k}<\frac{\delta }{4}%
<\varepsilon +\frac{\delta }{2}.
\end{equation*}%
Assume that the inequality (\ref{eqno9}) holds for some $n$. We show that
the inequality (\ref{eqno9}) holds for $n+1$. By the condition $(S2)$, we get%
\begin{equation*}
\mathcal{S}\left( x_{k},x_{k},x_{k+n+1}\right) \leq 2\mathcal{S}\left(
x_{k},x_{k},x_{k+1}\right) +\mathcal{S}\left(
x_{k+n+1},x_{k+n+1},x_{k+1}\right) \text{.}
\end{equation*}%
From Lemma \ref{lem5}, we have
\begin{equation*}
\mathcal{S}\left( x_{k+n+1},x_{k+n+1},x_{k+1}\right) =\mathcal{S}\left(
x_{k+1},x_{k+1},x_{k+n+1}\right)
\end{equation*}%
and so it suffices to prove
\begin{equation*}
\mathcal{S}\left( x_{k+1},x_{k+1},x_{k+n+1}\right) \leq \varepsilon \text{.}
\end{equation*}%
To do this, we show%
\begin{equation*}
M_{z}^{S}(x_{k},x_{k+n})\leq \varepsilon +\delta \text{.}
\end{equation*}%
Then we find%
\begin{equation*}
a\mathcal{S}(x_{k},x_{k},x_{k+n})<\mathcal{S}(x_{k},x_{k},x_{k+n})<%
\varepsilon +\frac{\delta }{2}\text{,}
\end{equation*}%
\begin{eqnarray*}
\frac{b}{2}\left[ \mathcal{S}(x_{k},x_{k},x_{k+1})+\mathcal{S}%
(x_{k+n},x_{k+n},x_{k+n+1})\right] &<&\mathcal{S}(x_{k},x_{k},x_{k+1})+%
\mathcal{S}(x_{k+n},x_{k+n},x_{k+n+1}) \\
&<&\frac{\delta }{4}+\frac{\delta }{4}=\frac{\delta }{2}
\end{eqnarray*}%
and%
\begin{eqnarray}
&&\frac{c}{2}\left[ \mathcal{S}(x_{k},x_{k},x_{k+n+1})+\mathcal{S}%
(x_{k+n},x_{k+n},x_{k+1})\right]  \notag \\
&\leq &\frac{c}{2}\left[ 4\mathcal{S}(x_{k},x_{k},x_{k+1})+\mathcal{S}%
(x_{k+1},x_{k+1},x_{k+1+n})+\mathcal{S}(x_{k},x_{k},x_{k+n})\right]  \notag
\\
&=&c\left[ 2\mathcal{S}(x_{k},x_{k},x_{k+1})+\frac{\mathcal{S}%
(x_{k+1},x_{k+1},x_{k+1+n})}{2}+\frac{\mathcal{S}(x_{k},x_{k},x_{k+n})}{2}%
\right]  \notag \\
&<&c\left[ \frac{\delta }{2}+\varepsilon +\frac{\delta }{2}\right]
<\varepsilon +\delta \text{.}  \label{eqno10}
\end{eqnarray}%
Using the definition of $M_{z}^{S}(x_{k},x_{k+n})$, the condition $(ii)$ and
the inequalities (\ref{eqno10}), we obtain%
\begin{equation*}
M_{z}^{S}(x_{k},x_{k+n})\leq \varepsilon +\delta
\end{equation*}%
and so%
\begin{equation*}
\mathcal{S}\left( x_{k+1},x_{k+1},x_{k+n+1}\right) \leq \varepsilon \text{.}
\end{equation*}%
Hence we get%
\begin{equation*}
\mathcal{S}(x_{k},x_{k},x_{k+n+1})<\varepsilon +\frac{\delta }{2}\text{,}
\end{equation*}%
whence $\{x_{n}\}$ is Cauchy. From the completeness hypothesis, there exists
a point $u\in X$ such that $x_{n}\rightarrow u$ for $n\rightarrow \infty $.
Also we get%
\begin{equation*}
\underset{n\rightarrow \infty }{\lim }Tx_{n}=\underset{n\rightarrow \infty }{%
\lim }x_{n+1}=u\text{.}
\end{equation*}%
Now we prove that $u$ is a fixed point of $T$. On the contrary, $u$ is not a
fixed point of $T$. Then using the condition $(i)$ and the property of $\phi
$, we obtain%
\begin{eqnarray*}
\mathcal{S}(Tu,Tu,Tx_{n}) &\leq &\phi (M_{z}^{S}(u,x_{n}))<M_{z}^{S}(u,x_{n})
\\
&=&\max \left\{
\begin{array}{c}
a\mathcal{S}(u,u,x_{n}),\frac{b}{2}\left[ \mathcal{S}(u,u,Tu)+\mathcal{S}%
(x_{n},x_{n},Tx_{n})\right] , \\
\frac{c}{2}\left[ \mathcal{S}(u,u,Tx_{n})+\mathcal{S}(x_{n},x_{n},Tu)\right]%
\end{array}%
\right\}
\end{eqnarray*}%
and so taking a limit for $n\rightarrow \infty $. Using Lemma \ref{lem5}, we
find%
\begin{equation*}
\mathcal{S}(Tu,Tu,u)<\max \left\{ \frac{b}{2}\mathcal{S}(u,u,Tu),\frac{c}{2}%
\mathcal{S}(u,u,Tu)\right\} <\mathcal{S}(Tu,Tu,u)\text{,}
\end{equation*}%
a contradiction. It should be $Tu=u$. We show that $u$ is a unique fixed
point of $T$. Let $v$ be another fixed point of $T$ such that $u\neq v$.
From the condition $(i)$ and Lemma \ref{lem5}, we have%
\begin{eqnarray*}
\mathcal{S}(Tu,Tu,Tv) &=&\mathcal{S}(u,u,v)\leq \phi
(M_{z}^{S}(u,v))<M_{z}^{S}(u,v) \\
&=&\max \left\{
\begin{array}{c}
a\mathcal{S}(u,u,v),\frac{b}{2}\left[ \mathcal{S}(u,u,Tu)+\mathcal{S}(v,v,Tv)%
\right] , \\
\frac{c}{2}\left[ \mathcal{S}(u,u,Tv)+\mathcal{S}(v,v,Tu)\right]%
\end{array}%
\right\} \\
&=&\max \left\{ a\mathcal{S}(u,u,v),c\mathcal{S}(u,u,v)\right\} <\mathcal{S}%
(u,u,v)\text{,}
\end{eqnarray*}%
a contradiction. So it should be $u=v$. Therefore, $T$ has a unique fixed
point $u\in X$.

Finally, we prove that $T$ is discontinuous at $u$ if and only if $\underset{%
x\rightarrow u}{\lim }M_{z}^{S}(x,u)\neq 0$. To do this, we can easily show
that $T$ is continuous at $u$ if and only if $\underset{x\rightarrow u}{\lim
}M_{z}^{S}(x,u)=0$. Suppose that $T$ is continuous at the fixed point $u$
and $x_{n}\rightarrow u$. Hence we get $Tx_{n}\rightarrow Tu=u$ and using
the condition $(S2)$, we find%
\begin{equation*}
\mathcal{S}(x_{n},x_{n},Tx_{n})\leq 2\mathcal{S}(x_{n},x_{n},u)+\mathcal{S}%
(Tx_{n},Tx_{n},u)\rightarrow 0\text{,}
\end{equation*}%
as $x_{n}\rightarrow u$. So we get $\underset{x_{n}\rightarrow u}{\lim }%
M_{z}^{S}(x_{n},u)=0$. On the other hand, let us consider $\underset{%
x_{n}\rightarrow u}{\lim }M_{z}^{S}(x_{n},u)=0$. Then we obtain $\mathcal{S}%
(x_{n},x_{n},Tx_{n})\rightarrow 0$ as $x_{n}\rightarrow u$, which implies $%
Tx_{n}\rightarrow Tu=u$. Consequently, $T$ is continuous at $u$.
\end{proof}

We give an example.

\begin{example}
\label{exm1} Let $X=\left\{ 0,2,4,8\right\} $ and $(X,\mathcal{S})$ be the $%
S $-metric space defined as in Example \ref{exm:S-metric}. Let us define the
self-mapping $T:X\rightarrow X$ as%
\begin{equation*}
Tx=\left\{
\begin{array}{ccc}
4 & ; & x\leq 4 \\
2 & ; & x>4%
\end{array}%
\right. \text{,}
\end{equation*}%
for all $x\in \left\{ 0,2,4,8\right\} $. Then $T$ satisfies the conditions
of Theorem \ref{thm1} with $a=\frac{3}{4},b=c=0$ and has a unique fixed
point $x=4$. Indeed, we get the following table$:$%
\begin{equation*}
\begin{array}{ccc}
\mathcal{S}\left( Tx,Tx,Ty\right) =0 & \text{and} & 3\leq M_{z}^{S}\left(
x,y\right) \leq 6\text{ when }x,y\leq 4 \\
\mathcal{S}\left( Tx,Tx,Ty\right) =4 & \text{and} & 6\leq M_{z}^{S}\left(
x,y\right) \leq 12\text{ when }x\leq 4,y>4 \\
\mathcal{S}\left( Tx,Tx,Ty\right) =4 & \text{and} & 6\leq M_{z}^{S}\left(
x,y\right) \leq 12\text{ when }x>4,y\leq 4%
\end{array}%
.
\end{equation*}%
Hence $T$ satisfies the conditions of Theorem \ref{thm1} with%
\begin{equation*}
\phi (t)=\left\{
\begin{array}{ccc}
5 & ; & t\geq 6 \\
\frac{t}{2} & ; & t<6%
\end{array}%
\right.
\end{equation*}%
and%
\begin{equation*}
\delta \left( \varepsilon \right) =\left\{
\begin{array}{ccc}
6 & ; & \varepsilon \geq 3 \\
6-\varepsilon & ; & \varepsilon <3%
\end{array}%
\right. .
\end{equation*}
\end{example}

Now we give the following results as the consequences of Theorem \ref{thm1}.

\begin{corollary}
\label{cor2} Let $(X,\mathcal{S})$ be a complete $S$-metric space and $T$ a
self-mapping on $X$ satisfying the conditions

$i)$ $\mathcal{S}\left( Tx,Tx,Ty\right) <M_{z}^{S}\left( x,y\right) $ for
any $x,y\in X$ with $M_{z}^{S}\left( x,y\right) >0$,

$ii)$ There exists a $\delta =\delta \left( \varepsilon \right) >0$ such
that $\varepsilon <M_{z}^{S}\left( x,y\right) <\varepsilon +\delta $ implies
$\mathcal{S}\left( Tx,Tx,Ty\right) \leq \varepsilon $ for a given $%
\varepsilon >0$.

Then $T$ has a unique fixed point $u\in X$. Also, $T$ is discontinuous at $u$
if and only if $\underset{x\rightarrow u}{\lim }M_{z}^{S}\left( x,u\right)
\neq 0$.
\end{corollary}

\begin{corollary}
\label{cor3} Let $(X,\mathcal{S})$ be a complete $S$-metric space and $T$ a
self-mapping on $X$ satisfying the conditions

$i)$ There exists a function $\phi :\mathbb{R}^{+}\rightarrow \mathbb{R}^{+}$
such that $\phi (\mathcal{S}(x,x,y))<\mathcal{S}(x,x,y)$ and $\mathcal{S}%
(Tx,Tx,Ty)\leq \phi (\mathcal{S}(x,x,y))$,

$ii)$ There exists a $\delta =\delta \left( \varepsilon \right) >0$ such
that $\varepsilon <t<\varepsilon +\delta $ implies $\phi (t)\leq \varepsilon
$ for any $t>0$ and a given $\varepsilon >0$.

Then $T$ has a unique fixed point $u\in X$.
\end{corollary}

The following theorem shows that the power contraction of the type $M_{z}^{S}\left( x,y\right)$ allows also the possibility of discontinuity at the
fixed point.

\begin{theorem}
\label{thm3} Let $(X,\mathcal{S})$ be a complete $S$-metric space and $T$ a
self-mapping on $X$ satisfying the conditions

$i)$ There exists a function $\phi :\mathbb{R}^{+}\rightarrow \mathbb{R}^{+}$
such that $\phi (t)<t$ for each $t>0$ and
\begin{equation*}
\mathcal{S}\left( T^{m}x,T^{m}x,T^{m}y\right) \leq \phi \left(
M_{z}^{S^{\ast }}\left( x,y\right) \right) \text{,}
\end{equation*}%
where%
\begin{equation*}
M_{z}^{S^{\ast }}\left( x,y\right) =\max \left\{
\begin{array}{c}
a\mathcal{S}\left( x,x,y\right) ,\frac{b}{2}\left[ \mathcal{S}\left(
x,x,T^{m}x\right) +\mathcal{S}\left( y,y,T^{m}y\right) \right] , \\
\frac{c}{2}\left[ \mathcal{S}\left( x,x,T^{m}y\right) +\mathcal{S}\left(
y,y,T^{m}x\right) \right]
\end{array}%
\right\}
\end{equation*}%
for all $x,y\in X$,

$ii)$ There exists a $\delta =\delta \left( \varepsilon \right) >0$ such
that $\varepsilon <M_{z}^{S^{\ast }}\left( x,y\right) <\varepsilon +\delta $
implies $\mathcal{S}\left( T^{m}x,T^{m}x,T^{m}y\right) \leq \varepsilon $
for a given $\varepsilon >0$.

Then $T$ has a unique fixed point $u\in X$. Also, $T$ is discontinuous at $u$
if and only if $\underset{x\rightarrow u}{\lim }M_{z}^{S^{\ast }}\left(
x,u\right) \neq 0$.
\end{theorem}

\begin{proof}
By Theorem \ref{thm1}, the function $T^{m}$ has a unique fixed point $u$.
Hence we have%
\begin{equation*}
Tu=TT^{m}u=T^{m}Tu
\end{equation*}%
and so $Tu$ is another fixed point of $T^{m}$. From the uniqueness fixed
point, we obtain $Tu=u$, that is, $T$ has a unique fixed point $u$.
\end{proof}

We note that if the $S$-metric $\mathcal{S}$ generates a metric $d$ then we
consider Theorem \ref{thm1} on the corresponding metric space as follows:

\begin{theorem}
\label{thm4} Let $(X,d)$ be a complete metric space and $T$ a self-mapping
on $X$ satisfying the conditions

$i)$ There exists a function $\phi :\mathbb{R}^{+}\rightarrow \mathbb{R}^{+}$
such that $\phi (t)<t$ for each $t>0$ and
\begin{equation*}
d(Tx,Ty)\leq \phi \left( M_{z}\left( x,y\right) \right) \text{,}
\end{equation*}%
for all $x,y\in X$,

$ii)$ There exists a $\delta =\delta \left( \varepsilon \right) >0$ such
that $\varepsilon <M_{z}\left( x,y\right) <\varepsilon +\delta $ implies $%
d(Tx,Ty)\leq \varepsilon $ for a given $\varepsilon >0$.

Then $T$ has a unique fixed point $u\in X$. Also, $T$ is discontinuous at $u$
if and only if $\underset{x\rightarrow u}{\lim }M_{z}\left( x,u\right) \neq
0 $.
\end{theorem}

\begin{proof}
By the similar arguments used in the proof of Theorem \ref{thm1}, the proof
can be easily proved.
\end{proof}

\section{\textbf{An Application to the Fixed-Circle Problem}}

\label{sec:2} In this section, we investigate new solutions to the
fixed-circle problem raised by \"{O}zg\"{u}r and Ta\c{s} in \cite%
{Ozgur-Tas-malaysian} related to the geometric properties of the set $Fix(T)$
for a self mapping $T$ on an $S$-metric space $(X,\mathcal{S})$. Some
fixed-circle or fixed-disc results, as the direct solutions of this problem,
have been studied using various methods on a metric space or some
generalized metric spaces (see \cite{Mlaiki, Mlaiki-Axioms,
Ozgur-Tas-circle-thesis, Ozgur-Tas-Celik, ozgur-aip, Ozgur-simulation,
Pant-Ozgur-Tas, Belgium, Tas math, Tas, Tas-fbed}).

Now we recall the notions of a circle and a disc on an $S$-metric space as
follows:%
\begin{equation*}
C_{x_{0},r}^{S}=\left\{ x\in X:\mathcal{S}(x,x,x_{0})=r\right\}
\end{equation*}%
and%
\begin{equation*}
D_{x_{0},r}^{S}=\left\{ x\in X:\mathcal{S}(x,x,x_{0})\leq r\right\} ,
\end{equation*}%
where $r\in \lbrack 0,\infty )$ \cite{Ozgur-Tas-circle-thesis}, \cite%
{Sedgi-Shobe-Aliouche}.

If $Tx=x$ for all $x\in C_{x_{0},r}^{S}$ (resp. $x\in D_{x_{0},r}^{S}$) then
the circle $C_{x_{0},r}^{S}$ (resp. the disc $D_{x_{0},r}^{S}$) is called as
the fixed circle (resp. fixed disc) of $T$.

We begin the following definition.

\begin{definition}
\label{def2} A self-mapping $T$ is called an $\mathcal{S}$-Zamfirescu type $%
x_{0}$-mapping if there exists $x_{0}\in X$ and $a,b\in \left[ 0,1\right) $
such that%
\begin{equation*}
\mathcal{S}(Tx,Tx,x)>0\Longrightarrow \mathcal{S}(Tx,Tx,x)\leq \max \left\{
\begin{array}{c}
a\mathcal{S}(x,x,x_{0}), \\
\frac{b}{2}\left[ \mathcal{S}(Tx_{0},Tx_{0},x)+\mathcal{S}(Tx,Tx,x_{0})%
\right]
\end{array}%
\right\} \text{,}
\end{equation*}%
for all $x\in X$.
\end{definition}

We define the following number:%
\begin{equation}
\rho :=\inf \left\{ \mathcal{S}(Tx,Tx,x):Tx\neq x,x\in X\right\} .
\label{rho}
\end{equation}%
Now we prove that the set $Fix(T)$ contains a circle (resp. a disc) by means
of the number $\rho $.

\begin{theorem}
\label{thm2} If $T$ is an $\mathcal{S}$-Zamfirescu type $x_{0}$-mapping with
$x_{0}\in X$ and the condition
\begin{equation*}
\mathcal{S}(Tx,Tx,x_{0})\leq \rho
\end{equation*}
holds for each $x\in C_{x_{0},\rho }^{S}$ then $C_{x_{0},\rho }^{S}$ is a
fixed circle of $T$, that is, $C_{x_{0},\rho }^{S}\subset Fix(T)$.
\end{theorem}

\begin{proof}
At first, we show that $x_{0}$\ is a fixed point of $T$. On the contrary,
let $Tx_{0}\neq x_{0}$. Then we have $\mathcal{S}(Tx_{0},Tx_{0},x_{0})>0$.
By the definition of an $\mathcal{S}$-Zamfirescu type $x_{0}$-mapping and
the condition $(S1)$, we obtain
\begin{eqnarray*}
\mathcal{S}(Tx_{0},Tx_{0},x_{0}) &\leq &\max \left\{ a\mathcal{S}%
(x_{0},x_{0},x_{0}),\frac{b}{2}\left[ \mathcal{S}(Tx_{0},Tx_{0},x_{0})+%
\mathcal{S}(Tx_{0},Tx_{0},x_{0})\right] \right\} \\
&=&b\mathcal{S}(Tx_{0},Tx_{0},x_{0}),
\end{eqnarray*}%
a contradiction because of $b\in \left[ 0,1\right) $. This shows that $%
Tx_{0}=x_{0}$.

We have two cases:

\textbf{Case 1: }If\textbf{\ }$\rho =0$, then we get $C_{x_{0},\rho
}^{S}=\{x_{0}\}$ and clearly this is a fixed circle of $T$.

\textbf{Case 2:} Let $\rho >0$ and $x\in C_{x_{0},\rho }^{S}$ be any point
such that $Tx\neq x$. Then we have
\begin{equation*}
\mathcal{S}(Tx,Tx,x)>0
\end{equation*}
and using the hypothesis we obtain%
\begin{eqnarray*}
\mathcal{S}(Tx,Tx,x) &\leq &\max \left\{ a\mathcal{S}(x,x,x_{0}),\frac{b}{2}%
\left[ \mathcal{S}(Tx_{0},Tx_{0},x)+\mathcal{S}(Tx,Tx,x_{0})\right] \right\}
\\
&\leq &\max \left\{ a\rho ,b\rho \right\} <\rho ,
\end{eqnarray*}%
which is a contradiction with the definition of $\rho $. Hence it should be $%
Tx=x$ whence $C_{x_{0},\rho }^{S}$ is a fixed circle of $T$.
\end{proof}

\begin{corollary}
\label{cor1} If $T$ is an $\mathcal{S}$-Zamfirescu type $x_{0}$-mapping with
$x_{0}\in X$ and the condition
\begin{equation*}
\mathcal{S}(Tx,Tx,x_{0})\leq \rho
\end{equation*}
holds for each $x\in D_{x_{0},\rho }^{S}$ then $D_{x_{0},\rho }^{S}$ is a
fixed disc of $T$, that is, $D_{x_{0},\rho }^{S}\subset Fix(T)$.
\end{corollary}

Now we give an illustrative example to show the effectiveness of our results.

\begin{example}
\label{exm4} Let $X=%
\mathbb{R}
$ and $(X,\mathcal{S})$ be the $S$-metric space defined as in Example \ref%
{exm:S-metric}. Let us define the self-mapping $T:X\rightarrow X$ as%
\begin{equation*}
Tx=\left\{
\begin{array}{ccc}
x & ; & x\mathcal{\in }\left[ -3,3\right] \\
x+1 & ; & x\mathcal{\notin }\left[ -3,3\right]%
\end{array}%
\right. \text{,}
\end{equation*}%
for all $x\in
\mathbb{R}
$. Then $T$ is an $\mathcal{S}$-Zamfirescu type $x_{0}$-mapping with $%
x_{0}=0,a=\frac{1}{2}$ and $b=0$. Indeed, we get%
\begin{equation*}
\mathcal{S}(Tx,Tx,x)=2\left\vert Tx-x\right\vert =2>0\text{,}
\end{equation*}%
for all $x\in \left( -\infty ,-3\right) \cup \left( 3,\infty \right) $. So
we obtain%
\begin{eqnarray*}
\mathcal{S}(Tx,Tx,x) &=&2\leq \max \left\{ aS\left( x,x,0\right) ,\frac{b}{2}%
\left[ \mathcal{S}(0,0,x)+\mathcal{S}(x+1,x+1,0)\right] \right\} \\
&=&\frac{1}{2}.2\left\vert x\right\vert .
\end{eqnarray*}%
Also we have
\begin{equation*}
\rho =\inf \left\{ \mathcal{S}(Tx,Tx,x):Tx\neq x,x\in X\right\} =2
\end{equation*}%
and%
\begin{equation*}
\mathcal{S}(Tx,Tx,0)=\mathcal{S}(x,x,0)\leq 2\text{,}
\end{equation*}%
for all $x\in C_{0,2}^{S}=\left\{ x:\mathcal{S}(x,x,0)=2\right\} =\left\{
x:2\left\vert x\right\vert =2\right\} =\left\{ x:\left\vert x\right\vert
=1\right\} $. Consequently, $T$ fixes the circle $C_{0,2}^{S}$ and the disc $%
D_{0,2}^{S}$.
\end{example}

\textbf{Acknowledgement.} This work is financially supported by Balikesir
University under the Grant no. BAP 2018 /019 and BAP 2018 /021.

\end{document}